\documentclass[preprint,12pt]{elsarticle}

\journal{Journal of Geometry and Physics}

\usepackage{amsmath,amssymb,amsthm,mathtools}
\usepackage{graphicx}
\usepackage[protrusion=true,expansion=false]{microtype}
\usepackage[hidelinks]{hyperref}
\allowdisplaybreaks

\numberwithin{equation}{section}

\biboptions{sort&compress}

\newtheorem{theorem}{Theorem}[section]
\newtheorem{proposition}[theorem]{Proposition}
\newtheorem{lemma}[theorem]{Lemma}
\newtheorem{corollary}[theorem]{Corollary}
\theoremstyle{definition}
\newtheorem{definition}[theorem]{Definition}
\newtheorem{example}[theorem]{Example}
\theoremstyle{remark}
\newtheorem{remark}[theorem]{Remark}

\newcommand{\Ric}{\operatorname{Ric}}
\newcommand{\Sym}{\operatorname{Sym}}
\newcommand{\Skew}{\operatorname{Skew}}
\newcommand{\tr}{\operatorname{tr}}
\newcommand{\D}{\mathcal D}
\newcommand{\hRic}{\operatorname{Ric}^{\widehat\nabla}}
\newcommand{\hnabla}{\widehat\nabla}
\newcommand{\HS}{\mathrm{HS}}

\begin{document}

\begin{frontmatter}

\title{A Quantitative Reeb-Form Obstruction for Real Hypersurfaces of Strict Nearly K\"ahler Six-Manifolds}

\author[gnu]{Jae Won Lee}
\ead{leejaew@gnu.ac.kr}
\author[knu]{Chul Woo Lee\texorpdfstring{\corref{cor1}}{}}
\ead{mathisu@knu.ac.kr}
\cortext[cor1]{Corresponding author}

\affiliation[gnu]{organization={Department of Mathematics Education, Gyeongsang National University},
                  addressline={501 Jinju-daero}, city={Jinju}, postcode={52828},
                  country={Republic of Korea}}
\affiliation[knu]{organization={Department of Mathematics, Kyungpook National University},
                  addressline={80 Daehak-ro, Buk-gu}, city={Daegu}, postcode={41566},
                  country={Republic of Korea}}

\begin{abstract}
Let $M^5$ be an oriented real hypersurface of a six-dimensional strict nearly K\"ahler manifold $(\widetilde M^6,\widetilde g,J)$ of positive constant type $\alpha$, with induced Reeb field $\xi=-JN$ and Reeb one-form $\eta=g(\,\cdot\,,\xi)$. Writing $\mathcal D_p=\ker\eta_p$ and letting $B_p$ be the orthogonal projection to $\mathcal D_p$ of the shape operator restricted to $\mathcal D_p$, we prove the exact pointwise identity
\begin{equation*}
\begin{aligned}
\left\|(d\eta)_p|_{\Lambda^2\mathcal D_p}\right\|^2
=\frac12\left\|\phi B_p+B_p\phi\right\|_{\HS}^2+8\alpha.
\end{aligned}
\end{equation*}
Consequently $\left\|(d\eta)_p|_{\Lambda^2\mathcal D_p}\right\|^2\ge 8\alpha$, so $(d\eta)_p$ cannot vanish at any point. The proof uses the decomposition of the horizontal part of $d\eta$ into a $\phi$-commuting skew endomorphism and the normal nearly K\"ahler operator $G_N$, which anti-commutes with $\phi$ and satisfies $G_N^2=-\alpha I$. Equality occurs precisely when $\phi B_p+B_p\phi=0$. As a secondary consequence, the Yano semi-symmetric metric connection associated with $\xi$ satisfies
\begin{equation*}
\begin{aligned}
\left\|\operatorname{Skew}\operatorname{Ric}^{\widehat\nabla}_p|_{\Lambda^2\mathcal D_p}\right\|^2\ge18\alpha,
\end{aligned}
\end{equation*}
and hence its Ricci tensor is nowhere symmetric. The estimate is computed explicitly on geodesic hyperspheres of the standard nearly K\"ahler $S^6$.
\end{abstract}


\end{frontmatter}

\medskip

\section{Introduction}

Nearly K\"ahler manifolds form a distinguished class of non-K\"ahler almost Hermitian manifolds. If $(\widetilde M,\widetilde g,J)$ is almost Hermitian with Levi-Civita connection $\widetilde\nabla$, its nearly K\"ahler tensor is
\begin{equation*}
\begin{aligned}
G(X,Y)=(\widetilde\nabla_XJ)Y.
\end{aligned}
\end{equation*}
The nearly K\"ahler condition is equivalent to the skew-symmetry of $G$ in its first two variables. In real dimension six, a strict nearly K\"ahler structure has positive constant type on each connected component; this yields a particularly rigid pointwise algebra and an associated $SU(3)$-structure. We keep the type $\alpha>0$ explicit rather than normalizing it to one. The arguments below use only this six-dimensional algebraic input, not an ambient curvature classification; see Gray \cite{Gray1970,Gray1976} and Nagy \cite{Nagy2002a} for background.

The simply connected homogeneous strict nearly K\"ahler six-manifolds were classified by Butruille \cite{Butruille2005}. Up to homothety they are
\begin{equation*}
\begin{aligned}
S^6,\qquad S^3\times S^3,\qquad \mathbb CP^3,\qquad F_{1,2}(\mathbb C^3).
\end{aligned}
\end{equation*}
Homogeneity is not intrinsic to the six-dimensional nearly K\"ahler condition. Cohomogeneity-one constructions were developed by Podest\`a and Spiro \cite{PodestaSpiro2010}, and Foscolo and Haskins constructed complete inhomogeneous nearly K\"ahler structures on $S^6$ and $S^3\times S^3$ \cite{FoscoloHaskins2017}. This makes it natural to separate pointwise hypersurface phenomena from conclusions that depend on the curvature tensor of one homogeneous model.

Let $M^5\subset\widetilde M^6$ be an oriented real hypersurface with unit normal $N$. The ambient almost complex structure induces
\begin{equation*}
\begin{aligned}
\xi=-JN,\qquad \eta=g(\,\cdot\,,\xi),\qquad JX=\phi X+\eta(X)N.
\end{aligned}
\end{equation*}
If $A$ is the shape operator, differentiation of $\xi$ gives
\begin{equation*}
\begin{aligned}
\nabla_X\xi=\phi AX-G(X,N).
\end{aligned}
\end{equation*}
Thus the exterior derivative of the Reeb one-form reflects both the extrinsic geometry of the hypersurface and the intrinsic nearly K\"ahler torsion.

There is substantial literature on real hypersurfaces of the homogeneous nearly K\"ahler models. For the standard $S^6$, Hopf hypersurfaces were studied by Berndt, Bolton and Woodward \cite{BerndtBoltonWoodward1995}. For $S^3\times S^3$, commuting and Hopf conditions, as well as related rigidity questions, were investigated in \cite{HuYaoZhang2018,HuYao2020,Hu2021}. The anticommuting relation $A\phi+\phi A=0$ was studied for $S^6$ and $S^3\times S^3$ by Hu, Yao and Zhang \cite{HuYaoZhang2019}, while the twistor-space cases $\mathbb CP^3$ and $F_{1,2}(\mathbb C^3)$ were treated by Deschamps and Loubeau \cite{DeschampsLoubeau2021}. These results are primarily ambient-specific classification or nonexistence theorems.

Two earlier almost-contact results are especially relevant. Blair showed that the induced structure on the totally geodesic $S^5\subset S^6$ has Killing structure tensors and is not cosymplectic, and more generally that a compact orientable hypersurface of a nearly K\"ahler manifold of positive sectional curvature cannot be cosymplectic \cite{Blair1971}. This is a global class-level obstruction using compactness and curvature. More recently, Anarella and D'haene studied a two-sphere of almost contact structures with the common Reeb field $\xi=-JN$ on hypersurfaces of strict nearly K\"ahler six-manifolds and proved, among other results, that none of these structures is co-K\"ahler \cite{AnarellaDhaene2026}. Their result concerns the full co-K\"ahler structure equations. It does not provide a pointwise lower bound for the canonical Reeb form.

The purpose of the present paper is to isolate the Reeb-form contribution and obtain a quantitative pointwise obstruction. Let
\begin{equation*}
\begin{aligned}
\mathcal D_p=\ker\eta_p
\end{aligned}
\end{equation*}
be the four-dimensional almost-contact distribution, and let $B_p$ denote the orthogonal projection to $\mathcal D_p$ of $A|_{\mathcal D_p}$. Our main result is the exact identity
\begin{equation}
\begin{aligned}
\left\|(d\eta)_p|_{\Lambda^2\mathcal D_p}\right\|^2
=\frac12\left\|\phi B_p+B_p\phi\right\|_{\HS}^2+8\alpha.
\label{eq:intro-quantitative}
\end{aligned}
\end{equation}
In particular,
\begin{equation}
\begin{aligned}
\left\|(d\eta)_p|_{\Lambda^2\mathcal D_p}\right\|^2\ge8\alpha>0,
\qquad p\in M,
\label{eq:intro-main}
\end{aligned}
\end{equation}
so $(d\eta)_p\neq0$ at every point.

The mechanism behind \eqref{eq:intro-quantitative} is elementary but rigid. On $\mathcal D_p$, the skew endomorphism $\phi B_p+B_p\phi$ commutes with $\phi$. In contrast, the normal nearly K\"ahler operator
\begin{equation*}
\begin{aligned}
G_NX=G(X,N)
\end{aligned}
\end{equation*}
anti-commutes with $\phi$ and satisfies
\begin{equation*}
\begin{aligned}
G_N\phi=-\phi G_N,
\qquad
G_N^2=-\alpha I_{\mathcal D_p}.
\end{aligned}
\end{equation*}
The commuting and anti-commuting parts are orthogonal for the Hilbert--Schmidt inner product. Formula \eqref{eq:intro-quantitative} follows by applying this orthogonal decomposition to the horizontal part of $d\eta$.

The distinction from the cited class-level results is important. The estimate does not classify the hypersurface and does not assert that $\eta$ is a contact form. In dimension five the contact condition is $\eta\wedge(d\eta)^2\neq0$, which is strictly stronger than $d\eta\neq0$. What \eqref{eq:intro-main} provides is a local preliminary obstruction: any induced almost contact condition requiring $d\eta=0$ is excluded before further ambient-specific equations are imposed. The equality case in \eqref{eq:intro-quantitative} is also explicit: it occurs precisely when $\phi B_p+B_p\phi=0$.

As a secondary consequence, we consider the Yano semi-symmetric metric connection
\begin{equation*}
\begin{aligned}
\widehat\nabla_XY=\nabla_XY+\eta(Y)X-g(X,Y)\xi.
\end{aligned}
\end{equation*}
For an arbitrary $n$-dimensional Riemannian manifold carrying a unit vector field, this connection satisfies
\begin{equation*}
\begin{aligned}
\operatorname{Skew}\operatorname{Ric}^{\widehat\nabla}
=-\frac{n-2}{2}\,d\eta.
\end{aligned}
\end{equation*}
Hence in dimension five our main estimate yields
\begin{equation*}
\begin{aligned}
\left\|\operatorname{Skew}\operatorname{Ric}^{\widehat\nabla}_p|_{\Lambda^2\mathcal D_p}\right\|^2
\ge18\alpha,
\end{aligned}
\end{equation*}
so the connection Ricci tensor is nowhere symmetric. The skew-Ricci identity itself is general; the strict nearly K\"ahler contribution is the quantitative lower bound.

The paper is organized as follows. Section~2 records the strict nearly K\"ahler hypersurface algebra and gives explicit computations on the standard $S^6$. Section~3 proves the quantitative Reeb-form identity and its immediate almost-contact consequences. Section~4 derives the Yano skew-Ricci identity and the resulting quantitative Ricci asymmetry. Section~5 discusses explicit homogeneous examples, the relation with earlier results, and the precise scope of the theorem.

\section{Strict Nearly K\"ahler Hypersurface Algebra}

We collect the pointwise nearly K\"ahler identities needed for the obstruction and fix the hypersurface notation. The connection with torsion is postponed until after the main theorem.

Let $(\widetilde M^6,\widetilde g,J)$ be a six-dimensional strict nearly K\"ahler manifold. Thus $J^2=-I$, $\widetilde g(JX,JY)=\widetilde g(X,Y)$, and
\begin{equation}
\begin{aligned}
G(X,Y)=(\widetilde\nabla_XJ)Y.
\label{eq:defG}
\end{aligned}
\end{equation}
The nearly K\"ahler identities used below are
\begin{equation}
\begin{aligned}
G(X,Y)&=-G(Y,X), \label{eq:Gskew}
\end{aligned}
\end{equation}
\begin{equation}
\begin{aligned}
G(X,JY)&=-JG(X,Y), \label{eq:GJ}
\end{aligned}
\end{equation}
\begin{equation}
\begin{aligned}
\widetilde g(G(X,Y),Z)&=-\widetilde g(G(X,Z),Y). \label{eq:G3form}
\end{aligned}
\end{equation}
In dimension six there is a positive constant $\alpha$ on each connected component such that
\begin{equation*}
\begin{aligned}
\|G(X,Y)\|^2=\alpha\bigl(\|X\|^2\|Y\|^2-\widetilde g(X,Y)^2-\widetilde g(JX,Y)^2\bigr),
\end{aligned}
\end{equation*}
and equivalently one may use
\begin{equation}
\begin{aligned}
G(G(X,Y),Z)=\alpha\{&\widetilde g(X,Z)Y-\widetilde g(Y,Z)X \\
&-\widetilde g(JX,Z)JY+\widetilde g(JY,Z)JX\}.
\label{eq:GG}
\end{aligned}
\end{equation}
For the six-dimensional strict case, these identities are standard; see, for example, \cite{Gray1976,AnarellaDhaene2026}. We do not normalize $\alpha$ to one.

\begin{example}[The standard strict nearly K\"ahler six-sphere]\label{ex:standardS6}
Identify $\mathbb R^7$ with the imaginary Cayley numbers and let $\times$ denote the corresponding cross product. On the unit sphere
\begin{equation*}
\begin{aligned}
S^6=\{x\in\mathbb R^7:\langle x,x\rangle=1\},
\end{aligned}
\end{equation*}
the standard nearly K\"ahler structure is
\begin{equation*}
\begin{aligned}
J_xY=x\times Y,\qquad Y\in T_xS^6.
\end{aligned}
\end{equation*}
For the unit round metric this is a strict nearly K\"ahler structure of constant type $\alpha=1$; see Gray \cite{Gray1970} and the hypersurface descriptions in \cite{AnarellaDhaene2026,BerndtBoltonWoodward1995}. Differentiating $J_xY=x\times Y$ in $\mathbb R^7$ and projecting tangentially gives
\begin{equation}
\begin{aligned}
G(X,Y)=X\times Y-\langle X\times Y,x\rangle x.
\label{eq:S6G}
\end{aligned}
\end{equation}
Thus the general identities above can be checked directly in the octonionic model. We will use this sphere repeatedly below to give genuine hypersurface examples rather than only tangent-space illustrations.
\end{example}

Let $M^5\subset\widetilde M^6$ be an oriented real hypersurface with local unit normal $N$. Its Gauss and Weingarten formulas are
\begin{equation*}
\begin{aligned}
\widetilde\nabla_XY&=\nabla_XY+g(AX,Y)N,
\end{aligned}
\end{equation*}
\begin{equation}
\begin{aligned}
\widetilde\nabla_XN&=-AX, \label{eq:Weingarten}
\end{aligned}
\end{equation}
where $A$ is self-adjoint:
\begin{equation*}
\begin{aligned}
g(AX,Y)=g(X,AY).
\end{aligned}
\end{equation*}
\begin{definition}[Induced almost contact metric structure]\label{def:induced-acms}
Define
\begin{equation*}
\begin{aligned}
\xi=-JN,\qquad \eta(X)=g(X,\xi),
\end{aligned}
\end{equation*}
and decompose
\begin{equation*}
\begin{aligned}
JX=\phi X+\eta(X)N.
\end{aligned}
\end{equation*}
Then
\begin{equation*}
\begin{aligned}
\phi^2X&=-X+\eta(X)\xi,
\end{aligned}
\end{equation*}
\begin{equation*}
\begin{aligned}
\phi\xi&=0,\qquad \eta(\phi X)=0,
\end{aligned}
\end{equation*}
\begin{equation*}
\begin{aligned}
g(\phi X,\phi Y)&=g(X,Y)-\eta(X)\eta(Y).
\end{aligned}
\end{equation*}
We write
\begin{equation*}
\begin{aligned}
\D=\ker\eta.
\end{aligned}
\end{equation*}
Thus $TM=\D\oplus\operatorname{span}\{\xi\}$ and $\phi^2=-I$ on $\D$.
\end{definition}

\begin{example}[A genuine family of real hypersurfaces in the standard $S^6$]\label{ex:latitude-family}
Fix a unit vector $e_7\in\mathbb R^7$ and a constant $c\in(-1,1)$. Put
\begin{equation*}
\begin{aligned}
M_c=\{x\in S^6:\langle x,e_7\rangle=c\},\qquad s=\sqrt{1-c^2}.
\end{aligned}
\end{equation*}
The restriction of $f(x)=\langle x,e_7\rangle$ to $S^6$ has gradient $e_7-cx$ along $M_c$, whose norm is $s>0$. Hence $c$ is a regular value and $M_c$ is a smooth embedded $5$-dimensional hypersurface. It is the geodesic hypersphere of radius $r=\arccos c$ centered at $e_7$.

A unit normal field in $S^6$ is
\begin{equation*}
\begin{aligned}
N_c(x)=\frac{e_7-cx}{s}.
\end{aligned}
\end{equation*}
Indeed, if $X\in T_xM_c$, then $\langle X,e_7\rangle=\langle X,x\rangle=0$, so $\langle X,N_c\rangle=0$, while
\begin{equation*}
\begin{aligned}
\langle N_c,N_c\rangle=\frac{1-c^2}{1-c^2}=1.
\end{aligned}
\end{equation*}
Since $c$ and $s$ are constant along $M_c$,
\begin{equation*}
\begin{aligned}
\widetilde\nabla_XN_c=-\frac{c}{s}X,
\end{aligned}
\end{equation*}
and the Weingarten formula gives
\begin{equation*}
\begin{aligned}
A=\frac{c}{s}I=(\cot r)I.
\end{aligned}
\end{equation*}
The induced Reeb field is also explicit:
\begin{equation*}
\begin{aligned}
\xi_c=-J_xN_c=-\frac{x\times e_7}{s}.
\end{aligned}
\end{equation*}
Thus $\eta_c(X)=\langle X,\xi_c\rangle$ and $\phi_cX=J_xX-\eta_c(X)N_c$.

For $c=0$ we obtain the equatorial hypersphere
\begin{equation*}
\begin{aligned}
M_0=S^6\cap e_7^\perp\cong S^5.
\end{aligned}
\end{equation*}
Here $N=e_7$ and $A=0$, so $M_0$ is totally geodesic. With the standard Cayley basis chosen so that $e_1\times e_7=-e_6$, at the point $p=e_1$ one has
\begin{equation*}
\begin{aligned}
\xi_p=-e_1\times e_7=e_6,\qquad
\D_p=\operatorname{span}\{e_2,e_3,e_4,e_5\}.
\end{aligned}
\end{equation*}
This gives a concrete realization of every object in Definition~\ref{def:induced-acms}.
\end{example}

Differentiating $\xi=-JN$ gives the basic hypersurface formula. Indeed,
\begin{equation*}
\begin{aligned}
\widetilde\nabla_X\xi=-\widetilde\nabla_X(JN)=-(\widetilde\nabla_XJ)N-J\widetilde\nabla_XN.
\end{aligned}
\end{equation*}
Using \eqref{eq:defG} and \eqref{eq:Weingarten},
\begin{equation*}
\begin{aligned}
\widetilde\nabla_X\xi=-G(X,N)+JAX.
\end{aligned}
\end{equation*}
The tangential part of $JAX$ is $\phi AX$. Hence
\begin{equation}
\begin{aligned}
\boxed{\nabla_X\xi=\phi AX-G(X,N).}
\label{eq:nablaxi}
\end{aligned}
\end{equation}
\begin{definition}[Normal nearly K\"ahler operator]\label{def:GN}
For a chosen unit normal $N$, define
\begin{equation*}
\begin{aligned}
G_NX:=G(X,N).
\end{aligned}
\end{equation*}
\end{definition}

\begin{example}[The normal operator on the totally geodesic $S^5$]\label{ex:GN-equator}
Use the equatorial hypersurface $M_0\subset S^6$ from Example~\ref{ex:latitude-family} and the point $p=e_1$. Formula~\eqref{eq:S6G} and $N=e_7$ give
\begin{equation*}
\begin{aligned}
G_NX=X\times e_7-\langle X\times e_7,e_1\rangle e_1.
\end{aligned}
\end{equation*}
For the standard Cayley basis,
\begin{equation*}
\begin{aligned}
e_2\times e_7=e_5,\quad e_5\times e_7=-e_2,\quad
e_3\times e_7=e_4,\quad e_4\times e_7=-e_3.
\end{aligned}
\end{equation*}
Hence, on $\D_p=\operatorname{span}\{e_2,e_3,e_4,e_5\}$,
\begin{equation*}
\begin{aligned}
G_Ne_2=e_5,\quad G_Ne_5=-e_2,\quad
G_Ne_3=e_4,\quad G_Ne_4=-e_3.
\end{aligned}
\end{equation*}
In particular $G_N^2=-I$ at $p$. This is an actual computation on an embedded hypersurface, not an abstract tangent-space model.
\end{example}

\begin{lemma}\label{lem:GN}
For every $X\in\D$, one has $G_NX\in\D$. Moreover, on $\D$,
\begin{equation}
\begin{aligned}
G_N\phi=-\phi G_N,
\label{eq:anticommute}
\end{aligned}
\end{equation}
and
\begin{equation}
\begin{aligned}
G_N^2=-\alpha I_{\D}.
\label{eq:GNsquare}
\end{aligned}
\end{equation}
In particular, $G_N|_{\D}$ is invertible.
\end{lemma}

\begin{proof}

First, $G(X,N)$ is tangent to $M$. From \eqref{eq:G3form}, with $Y=Z=N$,
\begin{equation*}
\begin{aligned}
\widetilde g(G(X,N),N)=-\widetilde g(G(X,N),N),
\end{aligned}
\end{equation*}
so $\widetilde g(G(X,N),N)=0$.

It remains to show that $G(X,N)$ is orthogonal to $\xi$. Since $\xi=-JN$,
\begin{equation*}
\begin{aligned}
g(G(X,N),\xi)=-\widetilde g(G(X,N),JN).
\end{aligned}
\end{equation*}
Let $a=\widetilde g(G(X,N),JN)$. Using \eqref{eq:G3form} and then \eqref{eq:GJ},
\begin{equation*}
\begin{aligned}
a=-\widetilde g(G(X,JN),N)=\widetilde g(JG(X,N),N).
\end{aligned}
\end{equation*}
Because $J$ is skew-adjoint with respect to $\widetilde g$,
\begin{equation*}
\begin{aligned}
\widetilde g(JG(X,N),N)=-\widetilde g(G(X,N),JN)=-a.
\end{aligned}
\end{equation*}
Thus $a=-a$, so $a=0$. Hence $G_NX\perp\xi$ and therefore $G_NX\in\D$.

Now let $X\in\D$. Then $JX=\phi X$. Using \eqref{eq:Gskew} and \eqref{eq:GJ},
\begin{equation*}
\begin{aligned}
G_N(\phi X)=G(JX,N)=-G(N,JX)=JG(N,X)=-JG(X,N).
\end{aligned}
\end{equation*}
Since $G_NX\in\D$, the action of $J$ on $G_NX$ equals $\phi$. Consequently
\begin{equation*}
\begin{aligned}
G_N(\phi X)=-\phi G_NX,
\end{aligned}
\end{equation*}
which proves \eqref{eq:anticommute}.

Putting $Y=Z=N$ in \eqref{eq:GG}. For $X\in\D$ we have
\begin{equation*}
\begin{aligned}
\widetilde g(X,N)=0,\qquad \widetilde g(JX,N)=0,\qquad \widetilde g(JN,N)=0,
\end{aligned}
\end{equation*}
so
\begin{equation*}
\begin{aligned}
G(G(X,N),N)=-\alpha X.
\end{aligned}
\end{equation*}
The left-hand side is $G_N^2X$, proving \eqref{eq:GNsquare}. Since $\alpha>0$, the equation $G_NX=0$ implies $X=0$. Thus $G_N$ is invertible on $\D$.
\end{proof}

\begin{example}[An adapted horizontal frame on the equatorial $S^5$]\label{ex:GN-frame}
In Example~\ref{ex:GN-equator}, $\alpha=1$. Take $X=e_2\in\D_p$. Since
\begin{equation*}
\begin{aligned}
\phi e_2=e_1\times e_2=e_3,\qquad
G_Ne_2=e_5,\qquad
\phi G_Ne_2=\phi e_5=-e_4,
\end{aligned}
\end{equation*}
the four vectors
\begin{equation*}
\begin{aligned}
e_2,\quad e_3,\quad e_5,\quad -e_4
\end{aligned}
\end{equation*}
form an orthonormal basis of $\D_p$. Moreover,
\begin{equation*}
\begin{aligned}
G_N(\phi e_2)=G_Ne_3=e_4=-\phi(G_Ne_2),
\end{aligned}
\end{equation*}
which displays the anti-commutation relation of Lemma~\ref{lem:GN} directly. The same table verifies $G_N^2=-I_{\D_p}$.
\end{example}

\section{The Quantitative Reeb-Form Obstruction}

We now prove the main estimate. The key point is to work directly on the four-dimensional distribution $\mathcal D=\ker\eta$, where the two parts of $d\eta$ have opposite complex-linearity with respect to $\phi$.

Fix $p\in M$ and let
\begin{equation*}
\begin{aligned}
\pi_{\mathcal D}:T_pM\longrightarrow\mathcal D_p
\end{aligned}
\end{equation*}
be the orthogonal projection. Define the self-adjoint endomorphism
\begin{equation*}
\begin{aligned}
B_p=\pi_{\mathcal D}\circ A|_{\mathcal D_p}:\mathcal D_p\longrightarrow\mathcal D_p.
\end{aligned}
\end{equation*}
Self-adjointness follows because, for $X,Y\in\mathcal D_p$,
\begin{equation*}
\begin{aligned}
g(B_pX,Y)=g(AX,Y)=g(X,AY)=g(X,B_pY).
\end{aligned}
\end{equation*}

\begin{theorem}[Quantitative Reeb-form obstruction]\label{thm:reebform}
Let $M^5$ be an oriented real hypersurface of a six-dimensional strict nearly K\"ahler manifold of positive constant type $\alpha$. Then, at every $p\in M$,
\begin{equation}
\begin{aligned}
\boxed{
\left\|(d\eta)_p|_{\Lambda^2\mathcal D_p}\right\|^2
=\frac12\left\|\phi B_p+B_p\phi\right\|_{\HS}^2+8\alpha.}
\label{eq:quantitative-main}
\end{aligned}
\end{equation}
Consequently,
\begin{equation}
\begin{aligned}
\left\|(d\eta)_p|_{\Lambda^2\mathcal D_p}\right\|^2\ge8\alpha>0,
\label{eq:quantitative-lower}
\end{aligned}
\end{equation}
and hence $(d\eta)_p\neq0$ for every $p\in M$. Equality in \eqref{eq:quantitative-lower} holds if and only if
\begin{equation}
\begin{aligned}
\phi B_p+B_p\phi=0.
\label{eq:quantitative-equality}
\end{aligned}
\end{equation}
\end{theorem}

\begin{proof}
Since $\eta(Y)=g(Y,\xi)$ and $\nabla g=0$,
\begin{equation*}
\begin{aligned}
(\nabla_X\eta)(Y)=g(\nabla_X\xi,Y).
\end{aligned}
\end{equation*}
Because the Levi-Civita connection is torsion-free,
\begin{equation*}
\begin{aligned}
d\eta(X,Y)
&=(\nabla_X\eta)(Y)-(\nabla_Y\eta)(X)\\
&=g(\nabla_X\xi,Y)-g(\nabla_Y\xi,X).
\end{aligned}
\end{equation*}
Using \eqref{eq:nablaxi},
\begin{equation*}
\begin{aligned}
d\eta(X,Y)
={}&g(\phi AX,Y)-g(\phi AY,X)\\
&-g(G_NX,Y)+g(G_NY,X).
\end{aligned}
\end{equation*}
Now $A$ is self-adjoint and $\phi$ is skew-adjoint. Hence
\begin{equation*}
\begin{aligned}
g(\phi AY,X)=-g(A\phi X,Y),
\end{aligned}
\end{equation*}
while $G_N$ is skew-adjoint by \eqref{eq:G3form}. Therefore
\begin{equation}
\begin{aligned}
\boxed{d\eta(X,Y)=g((\phi A+A\phi)X,Y)-2g(G_NX,Y).}
\label{eq:deta}
\end{aligned}
\end{equation}

We now restrict to $X,Y\in\mathcal D_p$. By the definition of $B_p$, the components of $AX$ and $A\phi X$ in the $\xi$-direction do not contribute after applying $\phi$ or pairing with $Y$. Thus \eqref{eq:deta} becomes
\begin{equation}
\begin{aligned}
(d\eta)_p(X,Y)
=g(K_pX,Y),
\qquad
K_p:=L_p-2G_N,
\label{eq:K-horizontal}
\end{aligned}
\end{equation}
where
\begin{equation*}
\begin{aligned}
L_p:=\phi B_p+B_p\phi.
\end{aligned}
\end{equation*}
Both $L_p$ and $G_N$ are skew-adjoint on $\mathcal D_p$. Indeed,
\begin{equation*}
\begin{aligned}
L_p^*=(\phi B_p+B_p\phi)^*
=-B_p\phi-\phi B_p=-L_p,
\end{aligned}
\end{equation*}
and Lemma~\ref{lem:GN} gives the corresponding property for $G_N$.

Next, because $\phi^2=-I$ on $\mathcal D_p$,
\begin{equation*}
\begin{aligned}
L_p\phi
&=\phi B_p\phi+B_p\phi^2
=\phi B_p\phi-B_p,\\
\phi L_p
&=\phi^2B_p+\phi B_p\phi
=-B_p+\phi B_p\phi.
\end{aligned}
\end{equation*}
Hence
\begin{equation}
\begin{aligned}
L_p\phi=\phi L_p.
\label{eq:Lcomm}
\end{aligned}
\end{equation}
On the other hand, Lemma~\ref{lem:GN} gives
\begin{equation}
\begin{aligned}
G_N\phi=-\phi G_N,
\qquad
G_N^2=-\alpha I_{\mathcal D_p}.
\label{eq:GNanti-square-main}
\end{aligned}
\end{equation}
Thus $L_p$ is $\phi$-complex-linear and $G_N$ is $\phi$-anti-complex-linear.

We use the Hilbert--Schmidt inner product
\begin{equation*}
\begin{aligned}
\langle S,T\rangle_{\HS}=\operatorname{tr}(S^*T)
\end{aligned}
\end{equation*}
on $\operatorname{End}(\mathcal D_p)$. Since $\phi$ is orthogonal on $\mathcal D_p$, conjugation by $\phi$ preserves this inner product. Equations \eqref{eq:Lcomm} and \eqref{eq:GNanti-square-main} therefore give
\begin{equation*}
\begin{aligned}
\langle L_p,G_N\rangle_{\HS}
&=\langle \phi^{-1}L_p\phi,\phi^{-1}G_N\phi\rangle_{\HS}\\
&=\langle L_p,-G_N\rangle_{\HS}
=-\langle L_p,G_N\rangle_{\HS}.
\end{aligned}
\end{equation*}
Consequently,
\begin{equation}
\begin{aligned}
\langle L_p,G_N\rangle_{\HS}=0.
\label{eq:HSorth}
\end{aligned}
\end{equation}
Furthermore, because $G_N^*=-G_N$ and $\dim\mathcal D_p=4$,
\begin{equation}
\begin{aligned}
\|G_N\|_{\HS}^2
&=\operatorname{tr}(G_N^*G_N)
=-\operatorname{tr}(G_N^2)\\
&=4\alpha.
\label{eq:GN-HS}
\end{aligned}
\end{equation}
Combining \eqref{eq:K-horizontal}, \eqref{eq:HSorth}, and \eqref{eq:GN-HS}, we obtain
\begin{equation}
\begin{aligned}
\|K_p\|_{\HS}^2
&=\|L_p-2G_N\|_{\HS}^2\\
&=\|L_p\|_{\HS}^2+4\|G_N\|_{\HS}^2\\
&=\|L_p\|_{\HS}^2+16\alpha.
\label{eq:K-HS}
\end{aligned}
\end{equation}

Finally, if a two-form $\omega$ on a Euclidean vector space is represented by a skew-adjoint endomorphism $K$ through $\omega(X,Y)=g(KX,Y)$, then
\begin{equation}
\begin{aligned}
\|\omega\|^2=\frac12\|K\|_{\HS}^2.
\label{eq:twoform-HS}
\end{aligned}
\end{equation}
Applying \eqref{eq:twoform-HS} to the restriction of $(d\eta)_p$ to $\mathcal D_p$ and using \eqref{eq:K-HS} gives exactly \eqref{eq:quantitative-main}. The lower bound follows because the first term on the right-hand side is nonnegative, and equality holds precisely when $L_p=\phi B_p+B_p\phi=0$.
\end{proof}

\begin{remark}[Orientation convention]\label{rem:orientation}
The orientation is used only to choose a global unit normal and hence global fields $\xi$ and $\eta$. The estimate is pointwise and is unchanged by replacing $N$ with $-N$: then $\xi$, $\eta$, $G_N$, and $d\eta$ all change sign, whereas their norms do not. Thus the same conclusion holds locally on a non-orientable hypersurface.
\end{remark}

\begin{remark}[Equality and the anticommuting shape condition]\label{rem:equality-hopf}
If the hypersurface is Hopf at $p$, then $A$ preserves $\mathcal D_p$ and $B_p=A|_{\mathcal D_p}$. In that case equality in \eqref{eq:quantitative-lower} is equivalent to
\begin{equation*}
\begin{aligned}
(\phi A+A\phi)|_{\mathcal D_p}=0.
\end{aligned}
\end{equation*}
This is the horizontal form of the anticommuting shape condition studied in the homogeneous models \cite{HuYaoZhang2019}. No Hopf assumption is needed for the estimate itself.
\end{remark}

\begin{example}[Geodesic hyperspheres in the standard $S^6$]\label{ex:theorem-latitudes}
Let $M_c$ be the family from Example~\ref{ex:latitude-family}, write $r=\arccos c$, and let $X\in\mathcal D$ be unit. Since $A=(\cot r)I$, we have
\begin{equation*}
\begin{aligned}
B=(\cot r)I,
\qquad
\phi B+B\phi=2(\cot r)\phi.
\end{aligned}
\end{equation*}
On the four-dimensional distribution $\mathcal D$, $\|\phi\|_{\HS}^2=4$, and the standard nearly K\"ahler sphere has $\alpha=1$. Theorem~\ref{thm:reebform} therefore gives
\begin{equation*}
\begin{aligned}
\left\|d\eta|_{\Lambda^2\mathcal D}\right\|^2
&=\frac12\,4\cot^2r\,\|\phi\|_{\HS}^2+8\\
&=8\cot^2r+8
=8\csc^2r.
\end{aligned}
\end{equation*}
For a direct component check, anti-commutation and skew-adjointness give $g(G_NX,\phi X)=0$, and hence
\begin{equation*}
\begin{aligned}
d\eta(X,\phi X)=2\cot r.
\end{aligned}
\end{equation*}
At the equator $r=\pi/2$, the quantitative lower bound is attained:
\begin{equation*}
\begin{aligned}
\left\|d\eta|_{\Lambda^2\mathcal D}\right\|^2=8.
\end{aligned}
\end{equation*}
At $p=e_1$, using $X=e_2$ and $G_NX=e_5$, one also obtains
\begin{equation*}
\begin{aligned}
(d\eta)_p(e_2,e_5)=-2.
\end{aligned}
\end{equation*}
Thus the estimate is realized by an explicit embedded hypersurface, and the equality case is nonempty.
\end{example}

\begin{corollary}\label{cor:notexact}
The Reeb one-form $\eta$ is not locally exact on any nonempty open subset of $M$. Equivalently, the Reeb vector field $\xi$ cannot be a local gradient vector field.
\end{corollary}

\begin{proof}
If $\eta=df$ on a nonempty open set, then $d\eta=d^2f=0$ there, contradicting Theorem~\ref{thm:reebform}. Since $\eta=g(\,\cdot\,,\xi)$, the equality $\eta=df$ is equivalent to $\xi=\operatorname{grad}f$ with respect to the induced metric.
\end{proof}

\begin{example}[Failure of local exactness on the equatorial $S^5$]\label{ex:notexact-equator}
On $M_0$ we have $(d\eta)_p(e_2,e_5)=-2$ at $p=e_1$. Hence no neighborhood of $p$ admits a function $f$ with $df=\eta$. Equivalently, the explicit Reeb field $\xi=-x\times e_7$ cannot be a gradient vector field on any neighborhood of $p$.
\end{example}

\begin{corollary}\label{cor:closedclasses}
The induced almost contact metric structure cannot satisfy, on any nonempty open subset, any geometric class whose defining equations include $d\eta=0$. In particular, it cannot be co-K\"ahler or almost cosymplectic in the standard sense $d\eta=d\Phi=0$, where $\Phi(X,Y)=g(X,\phi Y)$.
\end{corollary}

\begin{proof}
Both co-K\"ahler and almost cosymplectic structures require the Reeb one-form to be closed; see, for example, Blair \cite{Blair2010}. Theorem~\ref{thm:reebform} rules out this necessary condition pointwise.
\end{proof}

\begin{example}[The standard induced structure on the totally geodesic $S^5$]\label{ex:closedclass-equator}
For the equatorial $M_0\subset S^6$, the standard induced tensor $\phi=J-\eta\otimes N$ satisfies
\begin{equation*}
\begin{aligned}
\left\|d\eta|_{\Lambda^2\mathcal D}\right\|^2=8,
\qquad
(d\eta)_p(e_2,e_5)=-2.
\end{aligned}
\end{equation*}
Thus this actual induced structure is neither co-K\"ahler nor almost cosymplectic. This is compatible with the finer classification of Anarella--D'haene \cite{AnarellaDhaene2026}: on the totally geodesic $S^5$, the standard distinguished tensor is nearly cosymplectic, while another distinguished tensor is Sasakian. The present obstruction concerns the canonical Reeb one-form $\eta$, which is common to these structures.
\end{example}

\begin{remark}[What the estimate does not assert]\label{rem:not-contact}
The lower bound in Theorem~\ref{thm:reebform} should not be confused with the contact condition. In dimension five, $\eta$ is a contact form only when
\begin{equation*}
\begin{aligned}
\eta\wedge(d\eta)^2\neq0,
\end{aligned}
\end{equation*}
which requires maximal rank of $d\eta$ on $\mathcal D$. The theorem gives a positive norm for the horizontal part of $d\eta$, but it does not determine its rank. It also does not classify Sasakian, nearly Sasakian, or nearly cosymplectic hypersurfaces; those stronger structure questions require additional equations, as in \cite{AnarellaDhaene2026}.
\end{remark}

\section{The Yano Semi-Symmetric Metric Connection and Quantitative Ricci Asymmetry}

We now turn to a connection-theoretic consequence. This section is deliberately separated from the proof of Theorem~\ref{thm:reebform}: the Reeb-form estimate concerns the induced almost contact geometry and does not depend on a connection with torsion. The skew-Ricci formula below is universal for the Yano connection; the nearly K\"ahler input enters only through the quantitative estimate of Section~3.

\begin{definition}[Yano semi-symmetric metric connection]\label{def:yano}
Following Yano \cite{Yano1970}, the semi-symmetric metric connection associated with the unit Reeb field $\xi$ is
\begin{equation}
\begin{aligned}
\boxed{\hnabla_XY=\nabla_XY+\eta(Y)X-g(X,Y)\xi.}
\label{eq:ssmc}
\end{aligned}
\end{equation}
\end{definition}

Because $\nabla$ is torsion-free, the torsion of $\hnabla$ is
\begin{equation*}
\begin{aligned}
\widehat T(X,Y)
&=\hnabla_XY-\hnabla_YX-[X,Y]\\
&=\eta(Y)X-\eta(X)Y.
\end{aligned}
\end{equation*}
Metricity follows directly from $\nabla g=0$:
\begin{equation*}
\begin{aligned}
(\hnabla_Xg)(Y,Z)
={}&-\eta(Y)g(X,Z)+g(X,Y)\eta(Z)\\
&-\eta(Z)g(X,Y)+g(X,Z)\eta(Y)=0.
\end{aligned}
\end{equation*}

\begin{example}[The Yano torsion on the totally geodesic $S^5$]\label{ex:yano-equator}
On $M_0\subset S^6$, take $X=e_2$ at $p=e_1$. Then $\eta(X)=0$ and $\eta(\xi)=1$, so
\begin{equation*}
\begin{aligned}
\widehat T(X,\xi)=X.
\end{aligned}
\end{equation*}
Since $A=0$ and $G_Ne_2=e_5$,
\begin{equation*}
\begin{aligned}
\nabla_X\xi=-e_5,
\qquad
\hnabla_X\xi=e_2-e_5.
\end{aligned}
\end{equation*}
Thus the Yano connection is nontrivial even though the hypersurface is totally geodesic in $S^6$.
\end{example}

\begin{proposition}[Skew-Ricci identity for the Yano connection]\label{prop:skew-general}
Let $(M^n,g)$ carry a unit vector field $\xi$, let $\eta=g(\,\cdot\,,\xi)$, and define
\begin{equation*}
\begin{aligned}
\widehat\nabla_XY=\nabla_XY+\eta(Y)X-g(X,Y)\xi.
\end{aligned}
\end{equation*}
Then
\begin{equation}
\begin{aligned}
\boxed{\operatorname{Skew}\widehat{\operatorname{Ric}}=-\frac{n-2}{2}\,d\eta.}
\label{eq:skew-general}
\end{aligned}
\end{equation}
\end{proposition}

\begin{proof}
Write
\begin{equation*}
\begin{aligned}
\widehat\nabla_XY=\nabla_XY+Q(X,Y),
\qquad
Q(X,Y)=\eta(Y)X-g(X,Y)\xi.
\end{aligned}
\end{equation*}
For two connections differing by a $(1,2)$-tensor $Q$,
\begin{equation*}
\begin{aligned}
\widehat R(X,Y)Z={}&R(X,Y)Z+(\nabla_XQ)(Y,Z)-(\nabla_YQ)(X,Z)\\
&+Q(X,Q(Y,Z))-Q(Y,Q(X,Z)).
\end{aligned}
\end{equation*}
Because $\eta(Z)=g(Z,\xi)$ and $\nabla g=0$,
\begin{equation*}
\begin{aligned}
(\nabla_XQ)(Y,Z)=(\nabla_X\eta)(Z)Y-g(Y,Z)\nabla_X\xi.
\end{aligned}
\end{equation*}
Moreover,
\begin{equation*}
\begin{aligned}
Q(X,Q(Y,Z))={}&\{\eta(Y)\eta(Z)-g(Y,Z)\}X\\
&-\{\eta(Z)g(X,Y)-g(Y,Z)\eta(X)\}\xi.
\end{aligned}
\end{equation*}
Substitution gives
\begin{equation}
\begin{aligned}
\widehat R(X,Y)Z={}&R(X,Y)Z+(\nabla_X\eta)(Z)Y-(\nabla_Y\eta)(Z)X\\
&-g(Y,Z)\nabla_X\xi+g(X,Z)\nabla_Y\xi\\
&+\{\eta(Y)\eta(Z)-g(Y,Z)\}X\\
&-\{\eta(X)\eta(Z)-g(X,Z)\}Y\\
&+\{g(Y,Z)\eta(X)-g(X,Z)\eta(Y)\}\xi.
\label{eq:general-curv-ssmc}
\end{aligned}
\end{equation}
Choose a local orthonormal frame $\{e_1,\dots,e_n\}$ and use
\begin{equation*}
\begin{aligned}
\widehat\Ric(X,Y)=\sum_{i=1}^n g(\widehat R(e_i,X)Y,e_i).
\end{aligned}
\end{equation*}
Contracting \eqref{eq:general-curv-ssmc} term by term yields
\begin{equation}
\begin{aligned}
\widehat\Ric(X,Y)
={}&\Ric(X,Y)-(n-2)(\nabla_X\eta)(Y)\\
&-(\operatorname{div}\xi+n-2)g(X,Y)
+(n-2)\eta(X)\eta(Y).
\label{eq:general-ricci-ssmc}
\end{aligned}
\end{equation}
For completeness, the derivative terms contribute
\begin{equation*}
\begin{aligned}
-(n-2)(\nabla_X\eta)(Y)-g(X,Y)\operatorname{div}\xi,
\end{aligned}
\end{equation*}
and the remaining quadratic terms contribute
\begin{equation*}
\begin{aligned}
-(n-2)g(X,Y)+(n-2)\eta(X)\eta(Y).
\end{aligned}
\end{equation*}
The ordinary Ricci tensor, $g$, and $\eta\otimes\eta$ are symmetric. Hence
\begin{equation*}
\begin{aligned}
\widehat\Ric(X,Y)-\widehat\Ric(Y,X)
&=-(n-2)\{(\nabla_X\eta)(Y)-(\nabla_Y\eta)(X)\}\\
&=-(n-2)d\eta(X,Y).
\end{aligned}
\end{equation*}
Dividing by two proves \eqref{eq:skew-general}.
\end{proof}

\begin{example}[A numerical skew-Ricci value]\label{ex:skewric-equator}
On $M_0\subset S^6$, use $X=e_2$ and $G_NX=e_5$ at $p=e_1$. Since $A=0$,
\begin{equation*}
\begin{aligned}
\nabla_{e_2}\xi=-e_5,
\qquad
\nabla_{e_5}\xi=e_2,
\end{aligned}
\end{equation*}
and therefore
\begin{equation*}
\begin{aligned}
(d\eta)_p(e_2,e_5)=-2.
\end{aligned}
\end{equation*}
Proposition~\ref{prop:skew-general}, with $n=5$, gives
\begin{equation*}
\begin{aligned}
(\Skew\hRic)_p(e_2,e_5)=3.
\end{aligned}
\end{equation*}
\end{example}

\begin{proposition}\label{prop:ricci}
On a real hypersurface of a strict nearly K\"ahler six-manifold, the Ricci tensor of $\hnabla$ is
\begin{equation}
\begin{aligned}
\boxed{\hRic(X,Y)=\Ric(X,Y)-3g(\phi AX,Y)+3g(G_NX,Y)-3g(\phi X,\phi Y).}
\label{eq:hRic}
\end{aligned}
\end{equation}
\end{proposition}

\begin{proof}
Set $n=5$ in \eqref{eq:general-ricci-ssmc}. From \eqref{eq:nablaxi},
\begin{equation*}
\begin{aligned}
\operatorname{div}\xi=\tr(\phi A)-\tr(G_N).
\end{aligned}
\end{equation*}
Because $A$ is self-adjoint and $\phi$ is skew-adjoint,
\begin{equation*}
\begin{aligned}
\tr(\phi A)=0,
\end{aligned}
\end{equation*}
and $G_N$ is skew-adjoint, so $\tr(G_N)=0$. Thus $\operatorname{div}\xi=0$. Also,
\begin{equation*}
\begin{aligned}
(\nabla_X\eta)(Y)
=g(\nabla_X\xi,Y)
=g(\phi AX,Y)-g(G_NX,Y).
\end{aligned}
\end{equation*}
Substitution into \eqref{eq:general-ricci-ssmc} gives
\begin{equation*}
\begin{aligned}
\hRic(X,Y)={}&\Ric(X,Y)-3g(\phi AX,Y)+3g(G_NX,Y)\\
&-3g(X,Y)+3\eta(X)\eta(Y).
\end{aligned}
\end{equation*}
Using $g(\phi X,\phi Y)=g(X,Y)-\eta(X)\eta(Y)$ proves \eqref{eq:hRic}.
\end{proof}

\begin{proposition}\label{prop:decomp}
The symmetric and skew-symmetric parts of $\hRic$ are
\begin{equation}
\resizebox{0.90\linewidth}{!}{$
\begin{aligned}
\Sym\hRic(X,Y)
={}&\Ric(X,Y)-\frac32 g((\phi A-A\phi)X,Y)-3g(\phi X,\phi Y)
\end{aligned}
$}
\label{eq:symric}
\end{equation}
and
\begin{equation}
\begin{aligned}
\Skew\hRic(X,Y)
={}&-\frac32 g((\phi A+A\phi)X,Y)+3g(G_NX,Y).
\label{eq:skewric}
\end{aligned}
\end{equation}
Consequently,
\begin{equation}
\begin{aligned}
\hRic\ \text{is symmetric}
\quad\Longleftrightarrow\quad
\phi A+A\phi=2G_N.
\label{eq:symcondition}
\end{aligned}
\end{equation}
\end{proposition}

\begin{proof}
The tensors $\Ric$ and $(X,Y)\mapsto g(\phi X,\phi Y)$ are symmetric. For
\begin{equation*}
\begin{aligned}
C(X,Y)=g(\phi AX,Y),
\end{aligned}
\end{equation*}
we have $(\phi A)^*=-A\phi$, and hence
\begin{equation*}
\begin{aligned}
C(X,Y)+C(Y,X)&=g((\phi A-A\phi)X,Y),\\
C(X,Y)-C(Y,X)&=g((\phi A+A\phi)X,Y).
\end{aligned}
\end{equation*}
Since $G_N$ is skew-adjoint, \eqref{eq:symric} and \eqref{eq:skewric} follow from \eqref{eq:hRic}. The equivalence \eqref{eq:symcondition} follows by setting the skew part equal to zero.
\end{proof}

\begin{corollary}[Quantitative Ricci asymmetry]\label{thm:main}
For the Yano semi-symmetric metric connection \eqref{eq:ssmc},
\begin{equation*}
\begin{aligned}
\left\|\Skew\hRic_p|_{\Lambda^2\mathcal D_p}\right\|^2
&=\frac94\left\|(d\eta)_p|_{\Lambda^2\mathcal D_p}\right\|^2\\
&=\frac98\left\|\phi B_p+B_p\phi\right\|_{\HS}^2+18\alpha\\
&\ge18\alpha>0.
\end{aligned}
\end{equation*}
In particular, $\hRic$ is nowhere symmetric. Equality in the last inequality holds if and only if $\phi B_p+B_p\phi=0$.
\end{corollary}

\begin{proof}
Proposition~\ref{prop:skew-general} gives $\Skew\hRic=-\frac32d\eta$ in dimension five. Apply Theorem~\ref{thm:reebform} to the restriction to $\mathcal D_p$.
\end{proof}

\begin{example}[Ricci asymmetry on geodesic hyperspheres]\label{ex:decomp-latitude}
For $M_c\subset S^6$ with $r=\arccos c$, Example~\ref{ex:theorem-latitudes} gives
\begin{equation*}
\begin{aligned}
\left\|d\eta|_{\Lambda^2\mathcal D}\right\|^2=8\csc^2r.
\end{aligned}
\end{equation*}
Therefore
\begin{equation*}
\begin{aligned}
\left\|\Skew\hRic|_{\Lambda^2\mathcal D}\right\|^2=18\csc^2r.
\end{aligned}
\end{equation*}
At the equator this equals $18$, while at $r=\pi/4$ it equals $36$. At $p=e_1$ on the equator,
\begin{equation*}
\begin{aligned}
\hRic(e_2,e_5)=3,
\qquad
\hRic(e_5,e_2)=-3,
\end{aligned}
\end{equation*}
so the nonsymmetry is also visible componentwise.
\end{example}

\section{Examples, Comparisons, and Scope}

Theorem~\ref{thm:reebform} is local and does not depend on a homogeneous curvature formula. The classical homogeneous models nevertheless provide useful concrete settings in which to interpret the obstruction.

\begin{example}[A numerical latitude: $r=\pi/4$ in the standard $S^6$]\label{ex:pi4}
Choose
\begin{equation*}
\begin{aligned}
c=\frac1{\sqrt2},\qquad
M_{\pi/4}=\left\{x\in S^6:\langle x,e_7\rangle=\frac1{\sqrt2}\right\}.
\end{aligned}
\end{equation*}
Here $A=I$, $\alpha=1$, and $B=I$. Hence
\begin{equation*}
\begin{aligned}
\left\|d\eta|_{\Lambda^2\mathcal D}\right\|^2=16,
\qquad
\left\|\Skew\hRic|_{\Lambda^2\mathcal D}\right\|^2=36.
\end{aligned}
\end{equation*}
For every unit $X\in\mathcal D$ one also has
\begin{equation*}
\begin{aligned}
d\eta(X,\phi X)=2,
\qquad
(\Skew\hRic)(X,\phi X)=-3.
\end{aligned}
\end{equation*}
This concrete hypersurface shows that the lower bound of Theorem~\ref{thm:reebform} need not be an equality away from the equator.
\end{example}

\begin{remark}[Application to the homogeneous nearly K\"ahler $S^3\times S^3$]
The standard homogeneous nearly K\"ahler product has an additional almost product structure $P$, which plays a central role in its detailed hypersurface theory \cite{Hu2021,HuYao2020,HuYaoZhang2018}. The proof of Theorem~\ref{thm:reebform}, however, does not use $P$. Thus every oriented real hypersurface of the homogeneous nearly K\"ahler $S^3\times S^3$ satisfies
\begin{equation*}
\begin{aligned}
(d\eta)_p\neq0
\end{aligned}
\end{equation*}
at every point. In classification problems this supplies an ambient-independent preliminary test: any candidate family whose induced structure would force $d\eta=0$ is excluded before the $P$-dependent equations are considered.
\end{remark}

\begin{remark}[Application to the twistor-space cases]
The same pointwise obstruction applies to the standard strict nearly K\"ahler structures on
\begin{equation*}
\begin{aligned}
\mathbb CP^3=Sp(2)/(Sp(1)\times U(1))
\end{aligned}
\end{equation*}
and
\begin{equation*}
\begin{aligned}
F_{1,2}(\mathbb C^3)=SU(3)/T^2.
\end{aligned}
\end{equation*}
The twistor-space hypersurface geometry of these examples has been studied by Deschamps and Loubeau \cite{DeschampsLoubeau2021}. The present result is complementary: it uses no twistor-specific curvature formula and rules out a closed induced Reeb form before any finer classification condition is imposed.
\end{remark}

\begin{remark}[Relation with earlier co-K\"ahler and cosymplectic obstructions]
Blair's classical result excludes compact orientable cosymplectic hypersurfaces under a positive sectional-curvature assumption on the nearly K\"ahler ambient manifold \cite{Blair1971}. Anarella--D'haene prove that the distinguished family of almost contact structures with Reeb field $\xi=-JN$ on a strict nearly K\"ahler six-manifold is never co-K\"ahler \cite{AnarellaDhaene2026}. Theorem~\ref{thm:reebform} addresses a narrower tensorial feature but does so pointwise and without compactness, homogeneity, or a sectional-curvature hypothesis. More precisely, it gives the quantitative lower bound $\|(d\eta)_p|_{\Lambda^2\mathcal D_p}\|^2\ge8\alpha$. The result is therefore complementary to, rather than a replacement for, those richer structure theorems.
\end{remark}

\begin{remark}[Why the result is not restricted to homogeneous examples]
Butruille's classification \cite{Butruille2005} lists the four simply connected homogeneous strict nearly K\"ahler six-manifolds, while Foscolo--Haskins \cite{FoscoloHaskins2017} show that strict nearly K\"ahler geometry also admits complete inhomogeneous examples. Since the proof of Theorem~\ref{thm:reebform} is pointwise, the obstruction applies equally to those nonhomogeneous structures. Homogeneity is therefore useful for examples and subsequent classification problems, but not for the obstruction itself.
\end{remark}

The main conclusion is therefore deliberately focused but quantitative: strict nearly K\"ahler six-dimensional algebra forces a uniform pointwise contribution $8\alpha$ to the squared norm of the horizontal part of $d\eta$. The theorem does not classify the hypersurface, determine the rank of $d\eta$, or imply contactness. Its strength is that it is local, has an explicit equality condition, and precedes all model-dependent classification equations. The quantitative Ricci asymmetry of the Yano connection is a direct secondary consequence.

\medskip
\noindent\textbf{Curvature convention.} We use the curvature convention
\begin{equation*}
\begin{aligned}
R(X,Y)Z=\nabla_X\nabla_YZ-\nabla_Y\nabla_XZ-\nabla_{[X,Y]}Z
\end{aligned}
\end{equation*}
and
\begin{equation*}
\begin{aligned}
\Ric(X,Y)=\sum_{i=1}^{5}g(R(e_i,X)Y,e_i).
\end{aligned}
\end{equation*}
Stating these conventions explicitly is important because opposite curvature-sign conventions occur in the literature.

\section*{Data availability}
No data were used for the research described in this article.

\section*{Funding}
This research did not receive any specific grant from funding agencies in the public, commercial, or not-for-profit sectors.

\section*{Declaration of competing interest}
The authors declare that they have no known competing financial interests or personal relationships that could have appeared to influence the work reported in this paper.

\section*{Declaration of generative AI and AI-assisted technologies in the manuscript preparation process}
During the preparation of this work the author(s) used ChatGPT (OpenAI) for drafting assistance, language editing, manuscript organization, and LaTeX formatting. After using this tool, the authors reviewed and edited the content as needed and take full responsibility for the content of the article.

\section*{Acknowledgements}
The authors thank the editor and anonymous reviewers for their time and careful consideration of this manuscript.

\end{document}